\documentclass[11pt,reqno]{amsart}
\usepackage{graphicx}
\usepackage{float}
\usepackage[caption = false]{subfig}
\usepackage{enumerate}
\usepackage{mathtools}
\usepackage{breqn}
\usepackage{array, geometry, graphicx}
\usepackage{amsmath,amsfonts,paralist,amssymb,amsthm,mathrsfs}
\usepackage{pdfpages}
\newtheorem{theorem}{Theorem}[section]

\newtheorem{lemma}[theorem]{Lemma}

\theoremstyle{definition}

\theoremstyle{remark}

\numberwithin{equation}{section}

\begin{document}
	
\title[Hankel determinant]{Coefficient problems for certain Close-to-Convex Functions}
	\thanks{The first author is supported by Delhi Technological University, New Delhi}
	\author[Mridula Mundalia]{Mridula Mundalia}
	\address{Department of Applied Mathematics, Delhi Technological University, Delhi--110042, India}
	\email{mridulamundalia@yahoo.co.in}
	\author{S. Sivaprasad Kumar}
	\address{Department of Applied Mathematics, Delhi Technological University, Delhi--110042, India}
	\email{spkumar@dce.ac.in}

	
	\begin{abstract}
	In this paper, sharp bounds are established for the second Hankel determinant of logarithmic coefficients for normalised analytic functions satisfying certain differential inequality.  
	\end{abstract}
	
	\maketitle
	
	\section{Introduction}
Let $\mathcal{A}$ be the class of all analytic functions defined on the open unit disc $\mathbb{D}:=\{z\in\mathbb{C}:|z|<1\}$ with the Taylor series expansion of the form     $f(z)=z+\sum_{n=2}^{\infty}a_{n}z^{n}.$
Assume $\mathcal{S}\subset \mathcal{A}$ to be the class of univalent functions defined on $\mathbb{D}.$ A function $f\in\mathcal{S},$ lies in $\mathcal{S}^{*}$ if the domain $f(\mathbb{D})$ is starlike w.r.t origin. A function $f\in\mathcal{A}$ belongs to the class of close-to-convex functions $\mathcal{K}$  \cite{Kaplan(1952)}, if there exists $g\in\mathcal{S}^{*},$ such that $\operatorname{Re}(e^{i\beta}zf'(z)/g(z))>0$ for $z\in\mathbb{D}$ and $\beta\in(-\pi/2,\pi/2).$ Note that $\mathcal{S}^{*}\subset\mathcal{K}\subset\mathcal{S}.$ Moreover, for $\beta=0$ and specific choices of $g(z),$ namely $g(z)=1/(1-z),1/(1-z^{2}),1/(1-z+z^{2})$ and $1/(1-z)^{2},$ we obtain some special subclasses of close-to-convex functions $\mathcal{F}_{i}$ $(i=1,\ldots,4),$ defined as 
\begin{equation}\label{EQN40}
\begin{aligned}
   & \mathcal{F}_{1}:=\left\{f\in\mathcal{A}:\operatorname{Re} (1-z)f'(z)>0\right\}\\&
     \mathcal{F}_{2}:=\left\{f\in\mathcal{A}:\operatorname{Re}(1-z^{2})f'(z)>0\right\}\\&
          \mathcal{F}_{3}:=\left\{f\in\mathcal{A}:\operatorname{Re} (1-z+z^{2})f'(z)>0\right\}\\&
        \mathcal{F}_{4}:=\left\{f\in\mathcal{A}:\operatorname{Re} (1-z)^{2}f'(z)>0\right\},
\end{aligned}
\end{equation}
where $z\in\mathbb{D}.$
A function $f\in\mathcal{S}^{*}_{s}$ if for any $r<1$ sufficiently close to 1, and any $\gamma$ lying on the circle $|z|=r$, the angular velocity of $f(z)$ about the point $f(-\gamma)$ is positive at $\gamma$ as $z$ traverses the circle $|z|=r$ in the positive direction, i.e $\operatorname{Re}(2 z f'(z)/(f(z)-f(-\gamma)))>0$ for $|z|=r$ at  $z=\gamma.$ In 1959, Sakaguchi \cite{Sakaguchi(1959)} introduced and examined the class $\mathcal{S}^{*}_{s},$  consisting of functions starlike with respect to symmetric points, characterised by  
\begin{equation}\label{EQN51}
    \operatorname{Re} \left( \frac{2z f'(z)}{f(z) - f(-z)}\right) >0 \quad (z\in\mathbb{D}).
\end{equation} 
Note that Sakaguchi \cite{Sakaguchi(1959)} proved that $\mathcal{S}^{*}_{s}\subset\mathcal{K}.$ 
 The bounds of $\gamma_{n}$ for functions in $\mathcal{K}$ were examined in \cite[p. 116]{D.K. Thomas(2018)}, \cite{Thomas(2016)}. In 2018, Kumar and Vasudevrao \cite{Pranav Kumar &  Vasudevarao(2018)} obtained bounds on early logarithmic coefficients for the subclasses $\mathcal{F}_{1}, \mathcal{F}_{2}, \mathcal{F}_{3}$ of $\mathcal{K}.$ Infact in \cite{Ali & Vasudevrao(2018)} the estimates of logarithmic coefficients of class $\mathcal{F}_{4}\subset\mathcal{S}$ was discussed. Note that the conditions mentioned in classes $\mathcal{F}_{1},\mathcal{F}_{2}$ and $\mathcal{F}_{4}$ were considered by Ozaki \cite{Ozaki(1941)} as an important criteria of univalence. Moreover, it is important to mention that, the classes $\mathcal{F}_{2}$ and $\mathcal{F}_{4}$ have nice geometrical interpretations. Each member of $\mathcal{F}_{2}$ is convex in the direction of imaginary axis 
 and every $f\in\mathcal{F}_{4}$ is convex in the positive direction of real axis. 
 For $f\in\mathcal{S},$ define   $\widetilde{F}_{f}(z):=2 F_{f}(z),$ where 
\begin{equation}\label{EQN26}
    F_{f}(z)=\log \left(\frac{f(z)}{z}\right)= 2 \sum_{n=1}^{\infty} \gamma_{n}(f)z^{n} \quad (z\in\mathbb{D}\backslash\left\{0\right\}, \log 1:=0).
\end{equation} 
The coefficients $\gamma_{n}:=\gamma_{n}(f)$ in \eqref{EQN26} are known as logarithmic coefficients of $f.$ Sharp logarithmic coefficient estimates for the class $\mathcal{S},$ are already known for $n=1$ and $n=2,$ given by $|\gamma_{1}|\leq 1$ and $|\gamma_{2}|\leq 1/2+1/e^{2}.$ However bound of $\gamma_{n}$ for $n\geq 3,$ is still an open problem. 
Logarithmic coefficients played a crucial role in Milin's conjecture\cite{Milin(1977)}, which states that if $f\in\mathcal{S},$ then 
\[\sum_{m=1}^{n}\sum_{k=1}^{n}\left(k|\gamma_{k}|^{2}-\frac{1}{k}\right)\leq 0.\]
For $q,n\in\mathbb{N},$ the $q^{th}$ Hankel determinant $H_{q,n}(f),$ for $f\in\mathcal{A}$ is defined as, 
\[H_{q,n}(f)=\begin{vmatrix}
a_{n} & a_{n+1} & \cdots & a_{n+q-1}\\
a_{n+1} & a_{n+2}& \cdots  & a_{n+q} \\
\vdots & \vdots & \vdots & \vdots \\
a_{n+q-1} & a_{n+q} & \cdots & a_{n+2(q - 1)}
\end{vmatrix}.\]
Hankel determinant of exponential polynomials was studied by Ehrenborg in \cite{Ehrenborg(2000)}. Hayman \cite{Hayman(2001)} studied some properties of Hankel transform of an integer sequence. Hankel determinants $H_{q,n}(f) $ are useful, in showing that a function of bounded characteristic
in $\mathbb{D},$ i.e., a function which is a ratio of two bounded analytic functions with its Laurent series around the origin having integral coefficients, is rational \cite{Cantor(1963)}. It is worth mentioning that the Hankel determinant $H_{q,n}(f)$ has been investigated by several authors to study its rate of growth as $n \to \infty.$ Additionally, the notion of Hankel determinant is extremely helpful in the study of power series with integral coefficients and in the theory of singularities \cite{Dienes(1957)}. Noonan and Thomas \cite{Noonan & Thomas (1976)} studied the growth rate of the second Hankel determinant of an areally mean p-valent function. Pommerenke \cite{Pommerenke(1966)} proved that the
Hankel determinants of univalent functions satisfy $|H_{q,n}(f)|<K n^{-(\frac{1}{2}+\beta)q+\frac{3}{2}},$ where $\beta>1/4000$ and $K$ depends only on $q.$  Later, Hayman \cite{Hayman(1968)} proved that $H_{2,n}(f)<A n^{1/2},$ where $A$ is an absolute constant, for areally mean univalent functions. One can observe that $H_{2,1}(f)=a_{3}-a_{2}^{2},$
is the Fekete-Szeg\"{o} functional, which is further generalised to $a_{3}-\mu a_{2}^{2},$ where $\mu$ is a complex number. In 1916, Bieberbach \cite{Goodman(1983)II} estimated $H_{2,1}(f)$ for the class $\mathcal{S}.$
Several authors have studied second Hankel determinant (see \cite{Lecko(2022),Lee & Ravi & Subramanian(2013),Noonan & Thomas (1976)}), denoted by $H_{2,2}(f)$ for certain subclasses of $\mathcal{A}.$ Many authors have studied the problem of calculating $\max_{f\in\mathbb{F}}|H_{2,2}(f)|$ for various subfamilies $\mathbb{F}\subset\mathcal{A}$ \cite{Janteng(2006),Kowalczyk & Lecko(2014),Lee & Ravi & Subramanian(2013)}. Infact Noor \cite{Noor(1992)} studied Hankel determinant for close-to-convex functions. For $f\in\mathcal{A},$ define the $q^{th}$ Hankel determinant $H_{q,n}(F_{f}),$ where $q,n\in\mathbb{N},$ with entries as logarithmic coefficients, 
\[H_{q,n}(F_{f})=\begin{vmatrix}
\gamma_{n} & \gamma_{n+1} & \cdots & \gamma_{n+q-1}\\
\gamma_{n+1} & \gamma_{n+2}& \cdots  & \gamma_{n+q} \\
\vdots & \vdots & \vdots & \vdots \\
\gamma_{n+q-1} & \gamma_{n+q} & \cdots & \gamma_{n+2(q - 1)}
\end{vmatrix}.\] 
It is worth mentioning that, the problem of studying Hankel determinant with entries as logarithmic coefficients, was coined by Kowalczyk et al. \cite{Lecko(2022)}. 
In this paper, an attempt is made to study the problem of establishing sharp bound of $H_{2,1}(F_{f})=\gamma_{1}\gamma_{3}-\gamma_{2}^{2},$ which has a striking resemblance to $H_{2,1}(f)=a_{2}a_{4}-a_{3}^{2},$ where $f$ lies in some subclass of $\mathcal{A}.$ Further, from \eqref{EQN26} logarithmic coefficients of $f\in\mathcal{S},$ can be expressed as 
\[\gamma_{1}=\frac{1}{2}a_{2}, \quad \gamma_{2} = \frac{1}{2}\left(a_{3} - \frac{1}{2}a_{2}^{2}\right),\quad \gamma_{3}= \frac{1}{4}\left(a_{4} - a_{2}a_{3} +\frac{1}{3}a_{2}^{3}\right).\]
Then for $f\in\mathcal{S}$ \[H_{2,1}(F_{f})=\gamma_{1}\gamma_{3} - \gamma_{2}^{2} = \frac{1}{4}\left(a_{2}a_{4} - a_{3}^{2} +\frac{1}{12}a_{2}^{4}\right).\] 
Note that for $f\in\mathcal{A},$ the functional $|H_{2,1}(F_{f})|=|\gamma_{1}\gamma_{3} - \gamma_{2}^{2}|$ is rotationally invariant. For this reason, it is sufficient to assume that the second coefficient $a_{2}$ is real. Several authors have studied and established varied results for these classes (see \cite{Cho & Lecko(2020),Cho & Lecko(2019)}). Based on the definition of each class in \eqref{EQN40} and \eqref{EQN51}, functions in these classes can be represented in terms of Carath\'{e}odory class $\mathcal{P},$ consisting of analytic functions $p(z),$ defined as
\begin{equation}\label{EQN53}
    p(z)=1+\sum_{n=1}^{\infty}c_{n}z^{n}
\end{equation}such that $\operatorname{Re} p(z) > 0.$ For the purpose of computing the bounds for $H_{2,1}(F_{f}),$ we use the coefficient formula given in \eqref{EQN21} for $c_{1}$ given in  \cite{Caratheodory(1907)},  for $c_{2}$ \cite[pg. 166]{Pommerenke(1975)} and the formula for coefficient $c_{3}$ due to Libera and Zlotkiewicz \cite{Libera(1982),Libera(1983)}. Moreover, results stated by Cho et al. \cite{Cho & Lecko(2019)} highlight the importance of functions given in \eqref{EQN35} and \eqref{EQN36}. Classes defined in \eqref{EQN40} and \eqref{EQN51}, have been extensively studied on different aspects earlier, namely early logarithmic and inverse logarithmic coefficients bounds \cite{Goel & Kumar(2019)}, successive inverse coefficients \cite{Virender(2021)}, Hermitian-Toeplitz determinants \cite{S & V Kumar(2022)}. However, no work has been carried out for Hankel determinants of logarithmic coefficients for functions lying in these classes. In this article, we will address this problem.  Lemmas stated below serve as a prerequisite for deriving our subsequent results.

\begin{lemma}\cite{Libera(1982),Libera(1983),Pommerenke(1975)}\label{l6}
If $p\in\mathcal{P}$ of the form $p(z)=1+c_{1}z+c_{2}z^{2}+c_{3}z^{3}+\ldots,$ with $c_{1}\geq 0,$ then 
\begin{align}
    &c_{1} = 2 \zeta_{1} \label{EQN21}\\&
    c_{2} = 2 \zeta_{1}^{2} + 2 (1 - \zeta_{1}^{2})\zeta_{2} \nonumber
    \end{align}
     and
     \begin{align} \label{EQN23}
    c_{3}  &= 2 \zeta_{1}^{3} + 4(1 - \zeta_{1}^{2}) \zeta_{1} \zeta_{2} - 2 (1 - \zeta_{1}^{2})\zeta_{1}\zeta_{2}^{2} + 2(1 - \zeta_{1}^{2}) (1-|\zeta_{2}|^{2})\zeta_{3}
    \end{align}
for some $\zeta_{1}\in[0,1]$ and $\zeta_{2},\zeta_{3}\in\overline{\mathbb{D}}:=\left\{z\in\mathbb{C}:|z|<1\right\}.$ For $\zeta_{1}\in\mathbb{D}$ and $\zeta_{2}\in\mathbb{T}=\left\{z\in\mathbb{C}:|z|=1\right\},$ there is a unique function $p\in\mathcal{P}$ with $c_{1}$ and $c_{2}$ as in \eqref{EQN21}-\eqref{EQN23}, namely 
\begin{equation}\label{EQN35}
    p(z) = \frac{1+(\overline{\zeta_{1}}\zeta_{2}+\zeta_{1})z + \zeta_{2}z^{2}}{1+(\overline{\zeta_{1}}\zeta_{2}-\zeta_{1})z - \zeta_{2}z^{2}} \quad (z\in\mathbb{D}).
\end{equation}
For $\zeta_{1},\zeta_{2}\in\mathbb{D}$ and $\zeta_{3}\in \mathbb{T},$ there exists a unique $p\in\mathcal{P}$ with $c_{1},c_{2}$ and $c_{3}$ as in \eqref{EQN21}-\eqref{EQN23}, namely,
\begin{equation}\label{EQN36}
    \frac{1+(\overline{\zeta _1}\zeta _2  + \overline{\zeta _2}\zeta _3 + \zeta _1)z + (\overline{\zeta _1}\zeta _3 + \zeta _1\overline{\zeta _2} \zeta _3 +\zeta _2)z^2 + \zeta _3 z^3}{1+(\overline{\zeta _1}\zeta _2  + \overline{\zeta _2}\zeta _3 - \zeta _1)z + (\overline{\zeta _1}\zeta _3 - \zeta _1\overline{\zeta _2} \zeta _3 - \zeta _2)z^2 - \zeta _3 z^3} \quad (z\in\mathbb{D}).
\end{equation}
\end{lemma}

Following result is due to  Choi, Kim and Sugawa
\cite{Choi & Sugawa(2007)}.
\begin{lemma}\cite{Choi & Sugawa(2007)}\label{l7}
For $A,B,C\in\mathbb{R},$ let 
\[Y(A,B,C):= \max\left\{|A+B z + C z^{2}|+1-|z|^{2}:z\in\overline{\mathbb{D}}\right\}.\]
\begin{enumerate}
\item If $A C \geq 0,$ then 
 \begin{align*}
  Y(A,B,C)=\begin{cases} 
    |A|+|B|+|C|,  &  |B|\geq 2(1 - |C|), \\  \ \\
    1+|A|+\dfrac{B^{2}}{4(1-|C|)},  &   |B|<2(1-|C|). 
     \end{cases}
      \end{align*}
\item If $A C <0,$ then    
\begin{align*}
  Y(A,B,C)=\begin{cases} 
    1 - |A|+\dfrac{|B|}{4(1 - |C|)},  &  -4 A C(C^{-2} - 1) \leq B^{2} \land |B|<2(1 - |C|), \\  \ \\ 
    1+|A|+\dfrac{B^{2}}{4(1+|C|)},  &   B^{2}< \min \{4(1+|C|^{2}), - 4 A C (C^{-2} - 1)\}, \\ \ \\
    R(A,B,C), & \text{otherwise}
     \end{cases}
      \end{align*}
\end{enumerate}
 where  \begin{align*}
  R(A,B,C)=\begin{cases} 
     |A| +  |B| - |C|,  &  |C|(|B|+4|A|) \leq |AB|, \\  \ \\
    -|A| + |B| + |C|,  &  |AB| \leq |C|(|B| - 4|A|)  , \\ \ \\
    (|C|+|A|)\sqrt{1-\dfrac{B^{2}}{4 A C}}, & otherwise.
     \end{cases}
      \end{align*} 
\end{lemma}

\section{Hankel Determinant of Logarithmic Coefficients}
In this section, we establish results pertaining to sharp estimates of $H_{2,1}(F_{f})$ for the classes $\mathcal{F}_{1},\mathcal{F}_{2},\mathcal{F}_{3},\mathcal{F}_{4}$ and $\mathcal{S}^{*}_{s}.$ We begin by determining sharp bounds of $H_{2,1}(F_{f})$ for $f\in \mathcal{S}^{*}_{s}.$  
\begin{theorem}\label{EQN37}
Let $F_{f}$ be given by \eqref{EQN26}. 
If $f\in\mathcal{S}^{*}_{s},$ then sharp bound on Hankel determinant for $F_{f}$ is given by 
\begin{equation}\label{EQN48}
    |H_{2,1}(F_{f})|\leq \frac{1}{4}.
\end{equation} 
Above inequality is sharp due to the function
\begin{equation}\label{EQN47}
    \tilde{f}(z)=\int_{0}^{z}(1+t^{
2})/(1-t^{2})^{2} dt.
\end{equation}
\end{theorem}
\begin{proof}
Suppose $f\in\mathcal{A}$ of the form $f(z)=z+a_{2}z^{2}+a_{3}z^{3}+a_{4}z^{4}+\cdots$ satisfy \begin{equation}\label{EQN39}
    \frac{2zf'(z)}{f(z)-f(-z)} = p(z),
\end{equation}  where $p\in\mathcal{P}$ is given by \eqref{EQN53}. It is a well-known fact that $\mathcal{P}$ is invariant under rotation, we assume $c_{1}\in[0,2].$ From equation \eqref{EQN39} we express coefficients of $f(z),$ $a_{i}$$'s$ $(i=2,3,4)$ in terms of $c_{i}$$'s$ $(i=1,2,3)$ as follows
\begin{align*}
    a_2=\frac{1}{2}c_1, \quad a_3 = \frac{1}{2}c_2,\quad a_4 = \frac{1}{8} \left(c_1 c_2-2 c_3\right).
\end{align*}
Lemma \ref{l6} yields the following expression in terms of $\zeta_{i}$ $'s,$ where $\zeta_{i}\in\overline{\mathbb{D}}$ $(i=1,2,3).$
\begin{align}\label{EQN41}
\mathcal{L}:=\gamma_{1}\gamma_{3}-\gamma_{2}^{2} &= \frac{1}{4} \left( a_4 a_2-a_3^2 +\frac{1}{12}a_2^4\right) \nonumber \\& =\frac{1}{768} (c_1^4+12 c_2 c_1^2-24 c_3 c_1-48 c_2^2)  \nonumber \\& =\frac{1}{48} (6 \zeta _2^2(5 \zeta _1^2-3 \zeta _1^4-2)-11 \zeta _1^4-6 \zeta _1 (1-\zeta _1^2) \zeta _3 (1-| \zeta _2| {}^2) \nonumber \\&\quad -30 \zeta _2 \zeta _1^2 (1-\zeta _1^2)).
\end{align}
Expression in \eqref{EQN41} leads to,
\[|\mathcal{L}|=
\begin{cases} 
   1/4,  &  \zeta_{1}=0, \\  
   11/48,  &   \zeta_{1}=1. 
     \end{cases}\]
For $\zeta_{1}\in(0,1)$ and inequality $|\zeta_{3}|\leq 1,$ the expression in \eqref{EQN41} together with Lemma \ref{l6} results in the following inequality
\begin{equation}\label{EQN52}
    |\mathcal{L}| \leq \frac{1}{8}  \zeta _{1}  (1-\zeta _1^2) \Psi(A,B,C) 
\end{equation} where 
\[\Psi(A,B,C):= |A+B \zeta _2 + C \zeta _2^2| +1- |\zeta _2|{}^2\]
with
\begin{equation}\label{EQN42}
    A= -\frac{11 \zeta _1^3}{6 \left(1-\zeta _1^2\right)}, \quad B = -5 \zeta _1, \quad C = 3 \zeta _1-\frac{2}{\zeta _1}.
\end{equation}

\noindent In view of Lemma \ref{l7}, we now examine the following cases based on the expressions of $A,B$ and $C$ given in \eqref{EQN42}.\\
\noindent\textbf{I.} Let $\zeta_{1}\in X=(0,\zeta^{*}],$ where $\zeta^{*}=\sqrt{2/3}.$ It is easy to verify that $A C\geq0.$ Since 
$\left| B\right| -2 (1-\left| C\right| ) = \left(4-\zeta_{1}(2+\zeta_{1})\right)/\zeta_{1}$ is an increasing function on $X,$ then  $|B|-2(1-|C|)\geq 5 \zeta^{*}-2 > 0.$ Thus on applying Lemma \ref{l6} to inequality \eqref{EQN52}, we have
\[|\mathcal{L}| \leq \frac{1}{8} \zeta _1 (1-\zeta _1^2) (\left| A\right| +\left| B\right| +\left| C\right| ) = \frac{1}{4}-\frac{\zeta _1^4}{48}\leq \frac{1}{4}. \]
\noindent \textbf{II.} Assume $X^{*}=X^{c} \cap (0,1).$ Observe that $AC<0$ for each $\zeta_{1}\in X^{*}.$ This case demands the following subcases. \\
\indent \textbf{A.} Simple calculation reveals that, for each $\zeta_{1}$ in $X^{*},$
\[
 \begin{aligned}
    \left\{
    \begin{array}{ll}
       T_{1}(\zeta_{1}):= |B|-2(1-|C|)=\dfrac{1}{\zeta_{1}}(11 \zeta _1^2-2 \zeta _1-4)>0  \\ \\
         T_{2}(\zeta_{1}):=-4 A C\left(\dfrac{1}{C^2}-1\right) - B^2 = \dfrac{27 \zeta _1^2 - 62}{\zeta _1^2 (6-9 \zeta _1^2)}<0.
    \end{array}
    \right.
    \end{aligned}\]
From above we conclude that the set $T_{1}(X^{*}) \cap T_{2}(X^{*})$ is empty.   
Thus according to Lemma \ref{l7}, this case fails to exist for any $\zeta_{1}\in X^{*}.$\\
\indent \textbf{B.} For $\zeta_{1}\in X^{*},$ the expressions 
$4 (1 + | C|)^2$ and $-4 A C (C^{-2}-1)$ become 
\[ 4 \text{ }T_{3}(\zeta_{1}) :=4 (1 + | C|)^2  \text{ and } T_{4}(\zeta_{1}):=-4 A C \left(\frac{1}{C^2}-1\right)= \frac{22 \zeta _1^2 ( 9 \zeta _1^2-4)}{3(3\zeta _1^2-2)},\] respectively, where $T_{3}(x):=(3 x^{2}+x-2)^{2}/x^{2}.$ It is easy to observe that  $T_{3}'(x)$ is non-vanishing on $X^{*}.$ Infact $T_{3}'(1)=20,$ therefore $T_{3}(x)<4.$ Since $\zeta^{*}>{\zeta ^{*}}^{2},$ this yields that $T_{4}(x)$ is positive on $X^{*}.$ Thus inequality below is false for any $\zeta_{1}\in X^{*},$   
\[B^{2}= 25 \zeta_{1}^{2}<\min \{4 \text{ } T_{3}(x), T_{4}(x)\} = 4 \text{ } T_{3}(x).\] 
\indent \textbf{C.} The inequality below holds for $\zeta_{1}\in X^{*},$ 
\[|C|(|B|+4|A|) - |A B| = -\frac{13 \zeta _1^4-62 \zeta _1^2+60}{6 \left(1-\zeta _1^2\right)}\leq 0\] if and only if $13 \zeta _1^4-62 \zeta _1^2+60\geq 0.$ Due to Lemma \ref{l6} and equations \eqref{EQN41}-\eqref{EQN52}, \[|\mathcal{L}| \leq \frac{1}{8} \zeta _1 (1-\zeta _1^2) \left(| A| + | B | - | C | \right)= \frac{1}{4}-\frac{\zeta _1^4}{48}\leq \frac{1}{4}.\]
\indent
\textbf{D.} Finally the expression below is true for each $\zeta_{1}\in X^{*},$ 
\[| A B| -| C|  (| B| -4 | A| ) = \frac{277 \zeta _1^4-238 \zeta _1^2+60}{6(1-\zeta _1^2)}>0.\] 
Summarizing \textbf{I} and \textbf{II} inequality \eqref{EQN37} follows. In Lemma \ref{l6}, on replacing $\zeta_{1}=0, \zeta_{2}=\zeta_{3}=1$ in \eqref{EQN36} we get $\tilde{p}(z)=(1+z^{2})/(1-z^{2})\in\mathcal{P}.$ The function $\tilde{f}\in\mathcal{A}$ given in \eqref{EQN47} satisfies $2z\tilde{f}'(z)/(\tilde{f}(z)-\tilde{f}(-z)) = (1+z^2)/(1-z^2).$ Hence the result is sharp.
\end{proof}
\begin{theorem}
Let $F_{f}$ be given by \eqref{EQN26}. 
If $f\in\mathcal{F}_{2},$ then sharp bound on Hankel determinant for $F_{f}$ is given by
\begin{align}\label{EQN31}
    |H_{2,1}(F_{f})| \leq \frac{1}{4}.
\end{align}  Above inequality is sharp due to the function given in Theorem \ref{EQN37}.
\end{theorem}
\begin{proof}
Since $f\in\mathcal{F}_{2},$ then there exists an analytic function $p\in\mathcal{P}$ defined in \eqref{EQN53} such that
\begin{equation}\label{EQN54}
    (1 - z^{2})f'(z)= p(z).
\end{equation} Then from \eqref{EQN54} we obtain first three coefficients of $f(z)$ in terms of $c_{i}$$'s$ $(i=1,2,3)$ as follows:
\[a_{2}= \frac{1}{2}c_{1}, \quad a_{3}= \frac{1}{3} \left(1+c_2\right), \quad a_{4}=\frac{1}{4} \left(c_1+c_3\right).\] Using Lemma \ref{l6}, we express $\gamma_{1}\gamma_{3}-\gamma_{2}^{2}$ in terms of $\zeta_{i}'s$ by replacing $c_{i}$$'s$ with $\zeta_{i}$$'s$ $(i=1,2,3),$ 

\begin{align}
\mathcal{M}:=\gamma_{1}\gamma_{3}- \gamma_{2}^{2}&=\frac{1}{144} \big(\zeta _1^4(5 -4 \zeta _2 + 2 \zeta _2^2) +2 \zeta _1^2 (1 + 10 \zeta _2 + 7 \zeta _2^2) -4 (1 + 2 \zeta _2){}^2\nonumber\\&\quad+18 \zeta _1 \zeta _3 (1-\zeta _1^2)(1-|\zeta _2| {}^2)  \big) \label{EQN30}
\end{align}

\noindent Due to the fact that $|\zeta_{2}|\leq 1,$ inequality \eqref{EQN30} gives

\[
|\mathcal{M}| =\begin{cases}
  1/48,  &  \zeta_{1}=1, \\  
   (1/36)(1+2 \zeta_{2})^{2} \leq  1/4,  &   \zeta_{1}=0. 
\end{cases}
\]

\noindent Since $|\zeta_{3}|\leq 1,$ then for $\zeta_{1}\in(0,1),$ we have
\begin{align}\label{EQN55}
    |\mathcal{M}|& \leq \frac{1}{8} \zeta _1 (1-\zeta _1^2) \Psi(A,B,C),
    \end{align}
where \begin{equation*}
  \Psi(A,B,C):= |A +B \zeta _2 + C \zeta _2^2| +1 -|\zeta _2| {}^2
\end{equation*}
with \begin{equation}\label{EQN45}
    A=\frac{5 \zeta _1^4+2 \zeta _1^2-4}{18 \zeta _1 \left(1-\zeta _1^2\right)},\quad B=-\frac{2}{9 \zeta _1}(4-\zeta _1^2), \quad C= -\frac{1}{9 \zeta _1}(8+\zeta _1^2).
\end{equation}
In light of Theorem \ref{l7} and the expressions of $A,B$ and $C$ given in \eqref{EQN45}, we divide our result in two cases based on the value of $AC.$ \\
\textbf{I.} Observe that $A C$ is non-negative over the interval $Y=(0,\zeta^{*}],$ where $\zeta^{*}=(1/5)\sqrt{\sqrt{21}-1}.$ Since $|B| -2 (1-| C |)=\frac{8}{3 \zeta _1}-2 \geq 0,$ for each $\zeta_{1}\in Y,$ therefore by applying Lemma \ref{l7} to
\eqref{EQN55} yields the following inequality 
\begin{align*}
    |\mathcal{M}|&\leq \frac{1}{8} \zeta _1 (1-\zeta _1^2) (\left| A\right| +\left| B\right| +\left| C\right| )= \frac{1}{48} (12-12 \zeta _1^2-\zeta _1^4):=S_{1}(\zeta_{1}).
\end{align*}
It is clearly evident that $S_{1}(x)$ is a decreasing function of $x,$ so $S_{1}(x)\to 1/4$ whenever $x\to 0.$ Thus for this case $|\mathcal{M}|\leq 1/4.$ \\
\textbf{II.} Let $Y^{*}=Y^{c}\cap [0,1).$ Clearly $AC<0$ for each $\zeta_{1}\in Y^{*}.$ For our suitability, we divide the computations in five subcases.\\
\indent \textbf{A.} The expression in \eqref{EQN56} is true for each $\zeta_{1}\in Y^{*},$  
\begin{equation}\label{EQN56}
    S_{1}(x):=-4 A C \left(\frac{1}{C^2}-1\right) - B^2 = \frac{2 (\zeta _1^4-106 \zeta _1^2-12)}{27 (8 + \zeta _1^2)} \leq 0.
\end{equation}
provided $\zeta _1^4-106 \zeta _1^2-12 \leq 0.$
Additionally, from Case \textbf{I} we have  $|B|-2(1-|C|)=:S_{2}(x)\geq 0.$ Thus we conclude that $S_{1}(Y^{*})\cap S_{2}(Y^{*})$ is an empty set. \\
\indent 
\textbf{B.} Moreover for $\zeta_{1}\in Y^{*},$ it is easy to verify that 
\[
 \begin{aligned}
    \left\{
    \begin{array}{ll}
        S_{3}(\zeta_{1}) := 4 ( 1 + | C |)^2 = \dfrac{4}{81 \zeta _{1}^{2}} (\zeta _1^{2}+9\zeta _1+8)^2 > 0 \\ \\
          S_{4}(\zeta_{1}) :=   - 4 A C \left(\dfrac{1}{C^{2}}-1\right)= \dfrac{2 (\zeta _1^2-64) (5 \zeta _1^4+2 \zeta _1^2-4)}{81 \zeta _1^2 (8 + \zeta _1^2)}\leq 0.
    \end{array}
    \right.
    \end{aligned}\]

\noindent Therefore 
$B^{2}=2 (\zeta _1^2 - 4)/9 \zeta _1  >  \min\left\{S_{3}(\zeta_{1}),S_{4}(\zeta_{1})\right\}=S_{4}(\zeta_{1}).$ This leads us to the next case.\\
\indent \textbf{C.}
Consider the following expression
\begin{align*}
    |C|(|B| + 4|A| - |A B|) = \frac{S_{5}(\zeta_{1})}{81 \zeta _1^3 (1-\zeta _1^2)},
\end{align*} where $S_{5}(x):=2 x^8+10 x^7+11 x^6+84 x^5-90 x^4+24 x^3+52 x^2-64 x+16.$ Observe that for each $x\in Y^{*},$ the polynomial $S_{5}(x)\leq 0$ is true if and only if $x^2 (2 x^6+21 x^4+76)\leq 6 (8+x^4).$ Infact 
 \[\displaystyle{\min_{x\in Y^{*}}(x^{2}(2 x^6+21 x^4+76))= 36/625(329 \sqrt{21}-419) \geq \max_{x\in Y^{*}}(6(8+x^{4}))=54}.\] Thus  for each $x\in Y^{*},$ $S_{5}(x)$ must be positive. \\
\indent\textbf{D.} Finally the following inequality 
\begin{align*}
|A B| - |C|(|B| - 4|A| ) &= -\frac{1}{81 \zeta _1^3 (1-\zeta _1^2)}(2 \zeta _1^8-10 \zeta _1^7-5 \zeta _1^6-68 \zeta _1^5 -10 \zeta _1^4-104 \zeta _1^3\\& \quad-12 \zeta _1^2+128 \zeta _1+16) \leq 0,
\end{align*}
is equivalent to 
\[\zeta _1^2 (10 \zeta _1^5+5 \zeta _1^4+68 \zeta _1^3+10 \zeta _1^2+104 \zeta _1+12)\leq 2 (\zeta _1^8+64 \zeta _1+8),\] provided
$\zeta^{*}<\zeta_{1}\leq \sqrt{\sqrt{249}-15}.$
Note that the function  $R(x)$ defined below is a decreasing function of $x$ $\left(\zeta^{*}<x\leq \sqrt{\sqrt{249}-15}\right).$ Thus Lemma \ref{l7} together with \eqref{EQN55}, gives
\begin{align*}
    |\mathcal{M}| &\leq \frac{1}{8} \zeta _1 (1-\zeta _1^2) (-\left| A\right| +\left| B\right| +\left| C\right|) \\&= \frac{2 \zeta _1^5-5 \zeta _1^4-34 \zeta _1^3-2 \zeta _1^2+32 \zeta _1+4}{144 \zeta _1} =: R(\zeta_{1}) \\ \quad &
    \leq \frac{1}{600} (169-29 \sqrt{21}) \approx 0.0601755.
    \end{align*}
\indent \textbf{E.} Further for $\sqrt{\sqrt{249}-15}<\zeta _1<1,$ inequality \eqref{EQN55} and Lemma \ref{l7}, yields 
 \begin{align*}
    |\mathcal{M}| &\leq  \frac{1}{8} \zeta _1 (1-\zeta _1^2) \left((\left| A\right| +\left| C\right| ) \sqrt{1-\frac{B^2}{4 A C}}\right) \\&= \frac{1}{16 \sqrt{3}}(2-\zeta _1^2)^2 \sqrt{\frac{\zeta _1^2 \left(\zeta _1^4+20 \zeta _1^2-12\right)}{5 \zeta _1^6+42 \zeta _1^4+12 \zeta _1^2-32}}\\& \leq 
   \frac{1}{8} (3 \sqrt{249}-47) \approx  0.0424002 .
    \end{align*} 
Based on all the above cases, we arrive at the required bound. Infact on substituting $\zeta_{1}=0, \zeta_{2}=\zeta_{3}=1$ in \eqref{EQN30}, equality in \eqref{EQN31} is achieved. Clearly $\tilde{f}\in\mathcal{A}$ defined in \eqref{EQN47} belongs to $\mathcal{F}_{2}$ such that
$(1 - z^{2})\tilde{f}'(z)=(1+z^{2})/(1-z^{2})$ and it works as the extremal function . 
\end{proof}
Here below, we provide some Caratheod\'{o}ry functions that will be relevant for our onward study.
\begin{equation}\label{EQN38}
    \begin{aligned}
    p_{i}(z)=\left\{
    \begin{array}{lll}
        \dfrac{1+0.0605378 z+z^2}{1-z^2} , & i=1 \\ \\
          \dfrac{1-z^2}{1-0.747551 z+z^2} , & i=2 \\ \\
         \dfrac{1+0.261084 z+0.261084 z^2+z^3}{1-0.501762 z+0.501762 z^2-z^3}, & i=3.
    \end{array}
    \right.
    \end{aligned}
\end{equation}

\begin{theorem}
Suppose $X(x)=-48 x^4-96 x^3-392 x^2+24 x+357$ and $F_{f}$ be given by \eqref{EQN26}. 
If $f\in\mathcal{F}_{1},$ then sharp bound on Hankel determinant for $F_{f}$ is given by
\begin{align}\label{EQN33}
|H_{2,1}(F_{f})| \leq X(\eta)/2304,
\end{align}
 where $\eta\approx 0.0302$ is the unique real root of the equation $X'(x)=0.$  This  result is sharp due to the extremal function
\begin{equation}\label{EQN46}
\tilde{f}(z)=\int_{0}^{z}p_{1}(t)/(1-t)dt,
\end{equation}
where $p_{1}(z)$ is defined in \eqref{EQN38}.
\end{theorem}
\begin{proof}
For each $f\in\mathcal{F}_{1},$ there exists a function $p\in\mathcal{P}$ defined in \eqref{EQN53} satisfying $(1-z)f'(z) = p(z).$ On comparing coefficients of like power terms, we obtain
\[a_{2}= \frac{1}{2} \left(1 + c_1\right), \quad a_{3}=\frac{1}{3} \left(1+c_1+c_2\right), \quad a_{4} = \frac{1}{4} \left(1+c_1+c_2+c_3\right).\]
With the assumption that $c_{1}\in [0,2],$ and applying Lemma \ref{l6}, we obtain,
\begin{align}
 |\mathcal{N}|:=|\gamma_{1}\gamma_{3} - \gamma_{2}^{2}| &= \frac{1}{2304}|72 (1+2 \zeta _1)(1+2 \zeta _3 (1-| \zeta _2| {}^2)+2 \zeta _1 (1+2 \zeta _2-\zeta _2^2) + 2 \zeta _2\nonumber \\& \quad +2\zeta _1^3 (1-\zeta _2)^2 -2 \zeta _1^2 (\zeta _3 (1-| \zeta _2| {}^2)+\zeta _2-1))+3 (1+2 \zeta _1)^4 \nonumber \\& \quad-64 (1+2( \zeta _1 + \zeta _2) + 2 \zeta _1^2 (1-\zeta _2) )^{2}|  \label{EQN32}.
\end{align}

\noindent \textbf{I.} In equation \eqref{EQN32}, on  substituting $\zeta_{1}=1,$ we get $|\mathcal{N}|=155/2304\approx 0.067274.$ \\
\textbf{II.} Since $|\zeta_{3}|\leq 1,$ then for $\zeta_{1}\in [0,1),$ we obtain 
\begin{equation}\label{EQN57}
|\mathcal{N}| \leq \frac{1}{16} (1+2 \zeta _1) (1-\zeta _1^2)\Psi(A,B,C), 
\end{equation} 
where \[\Psi(A,B,C):=|A +B \zeta _2 + C \zeta _2^2|+ 1-| \zeta _2|{}^2\]
with \[A:=\frac{11+56 \zeta _1-8 \zeta _1^2+16 \zeta _1^3+80 \zeta _1^4}{144 \left(1 + 2 \zeta _1\right) \left(1-\zeta _1^2\right)},\quad  B:=\frac{4\zeta _1^2+4\zeta _1-7}{9\left(1 + 2 \zeta _1\right)}\] and \[C:=-\frac{2 \zeta _1^2+9 \zeta _1+16}{9 (1+2 \zeta _1)}.\]
Based on the range of $\zeta_{1}$ and the values of $A,B$ and $C$ defined above, we find that $A C < 0.$ Further as a consequence of Lemma \ref{l7}, we split our discussions into the following subcases.\\
\indent  \textbf{A.} For $\zeta_{1}\in[0,1),$ consider the expression,
    \[\mathcal{V}_{1}:=-4 A C \left(\frac{1}{C^{2}}- 1\right) - B^{2} = \frac{T(\zeta_{1})}{108 \left(1+2 \zeta _1\right) \left(2 \zeta _1^2+9 \zeta _1+16\right)},\] where $T(x)=32 x^5+336^4-2720 ^3+856 x^2+846 x-1687.$  It is easy to verify that  $\max_{x\in[0,1)}T(x)=-1687,$ thus $T(x)\leq 0.$ Suppose $\zeta^{*}=(1/2)(2\sqrt{2}-1),$ then  for $\zeta_{1}\in W_{1}\cup W_{1}^{*},$ where $W_{1}:=[0,\zeta^{*}]$ and $W_{1}^{*}:=W_{1}^{c}\cap [0,1),$ define 
    \[\mathcal{V}_{2}:=|B| - 2 (1-|C| )=\begin{cases}
      \dfrac{21-22 \zeta _1}{18 \zeta _1+9},  & \zeta_{1} \in W_{1}, \\ & \\
      \dfrac{8 \zeta _1^2-14 \zeta _1+7}{18 \zeta _1+9}, & \zeta_{1} \in W_{1}^{*}.
    \end{cases}\]
From the expression of $\mathcal{V}_{2},$ one can observe that, if $\zeta_{1}\in W_{1},$ then $\mathcal{V}_{2}\leq 0$ leads to $\zeta _1\geq 21/22,$ this is not true. Clearly for $\zeta_{1}\in W_{1}^{*},$ it is easy to verify that $\mathcal{V}_{2}>0.$ Thus $\mathcal{V}_{1}\cap \mathcal{V}_{2}$ is an empty set. Therefore according to Lemma \ref{l7}, this case fails to hold for any value of $\zeta_{1}\in [0,1).$ \\
\indent \textbf{B.} For $\zeta_{1}\in [0,1),$ simple computations reveal that
\[\mathcal{V}_{3}:=4 (1+|C|)^2  = \frac{4 (2 \zeta _1^2+27 \zeta _1+25)^2}{81 (1+2 \zeta _1)^2}>0\] and 
\[\mathcal{V}_{4}:=-4 A C \left(\frac{1}{C^{2}}-1\right) = \frac{(4 \zeta _1^2+36 \zeta _1-175) (80 \zeta _1^4+16 \zeta _1^3-8 \zeta _1^2+56 \zeta _1+11)}{324 (2 \zeta _1+1)^2 (2 \zeta _1^2+9 \zeta _1+16)} < 0.\] Above inequalities yield, 
\begin{align*}
   B^{2}=\frac{(7-4 \zeta _1-4 \zeta_1^2)^2}{81 \left(2 \zeta _1+1\right){}^2}&>\min\left\{\mathcal{V}_{3},\mathcal{V}_{4}\right\} =\mathcal{V}_{4}.
\end{align*}
\indent \textbf{C.} Furthermore for $\zeta_{1}\in W_{1}\cup W_{1}^{*},$ we have the following expression 
\[
\mathcal{V}_{5}:= \left|C \right|(\left| B \right| + 4 \left| A\right| )-\left| A B\right|=
\begin{cases}
   \dfrac{Q_{1}(\zeta_{1})}{1296 \left(1 + 2 \zeta _1\right)^2 \left(1-\zeta _1^2\right)}, & \zeta_{1}\in W_{1}, \\ & \\
 \dfrac{Q_{2}(\zeta_{1})}{1296 \left(1 + 2 \zeta _1\right)^2 \left(1-\zeta _1^2\right)},  & \zeta_{1}\in W_{1}^{*}\cup \{\zeta^{*}\},  
\end{cases}\]
where $Q_{1}(x):=1088 x^6+4096 x^5+6352 x^4+576 x^3-1252 x^2+3616 x+2419 $ and $Q_{2}(x):=192 x^6+1920 x^5+4912 x^4+1792 x^3+4436 x^2+4344 x-1011.$ A computation reveals that $Q_{1}'(x)$ is never zero on $W_{1}$ and $Q_{1}'(0)=3616,$  thus $ \mathcal{V}_{5}\geq Q_{1}(0)/1296>0.$ Similarly when $\zeta_{1}\in W_{1}^{*},$ then $\mathcal{V}_{5}\to Q_{2}(\zeta^{*})/(8 \sqrt{2}-10)>0$ as $\zeta_{1} \to \zeta^{*}.$ Therefore in view of Lemma \ref{l7}, this case fails to hold any $\zeta_{1}\in[0,1).$\\
\indent \textbf{D.} Consider the expression
\[\mathcal{V}_{6}:=|A B| - |C|(|B| - 4 |A|) = \frac{Q_{2}(\zeta_{1})}{1296 \left(1+2 \zeta _1\right){}^2 \left(1-\zeta _1^2\right)}.\] 
It can be easily verified that $Q_{2}'(x)$ is non-vanishing over the range $[0,1).$ Infact $Q_{2}(0)=-1011$ and $Q_{2}(1/2)=2864,$ thus by intermediate value property their exists $\gamma\in (0,1/2)$ such that $Q_{2}(\gamma)=0.$ Infact $Q_{2}(\gamma)=0,$ provided $\gamma \approx 0.190991.$ Thus $Q_{2}(x)\leq 0$ for each $x\in[0,\gamma],$ this leads to $\mathcal{V}_{6}\leq Q_{2}(\gamma)/(1296(1+2\gamma)^{2}(1-\gamma^{2}))= 0.$ Consequently, on applying Lemma \ref{l7} to inequality \eqref{EQN57}, we have 
\begin{align*}
|\mathcal{N}| &\leq
\frac{1}{16} \left(1+2 \zeta _1\right) (1-\zeta _1^2) (\left| B\right| +\left| C\right| -\left| A\right|)\leq \phi(\eta) \quad (\eta \approx 0.0302689),
\end{align*}
where \[\phi(x):=\frac{X(x)}{2304}=\frac{-48 x^4-96 x^3-392 x^2+24 x + 357}{2304} \quad 0\leq x \leq \gamma.\]  Observe that for $0\leq x\leq \gamma,$
\[\phi'(x) = \frac{-192 x^3-288 x^2-784 x+24}{2304}=0, \] holds true only if $x= \eta < \gamma,$ also $\phi''(\eta)<0.$ Consequently, this leads us to the inequality $\phi(x)\leq   \phi(\eta) \approx 0.155106.$\\
\indent \textbf{E.} Furthermore, for $\gamma < \zeta_{1} < 1,$ from  \eqref{EQN57} and Lemma \ref{l7},  it follows that
\begin{align*}
|\mathcal{N}| & \leq \frac{1}{16} \left(1+2 \zeta _1\right) (1-\zeta _1^2) (\left| A\right| +\left| C\right| ) \sqrt{1-\frac{B^2}{4 A C}} \leq \psi(\gamma) \quad (\gamma \approx 0.190991), 
\end{align*}
where \begin{align*}
    \psi(x)&:= \frac{1}{768} (48 x^4-128 x^3-232 x^2+200 x+267)\\&\quad \times \sqrt{\frac{32 x^6+208 x^5+544 x^4+216 x^3+14 x^2+257 x+124}{480 x^6+2256 x^5+4224 x^4+888 x^3+1194 x^2+2985 x+528}}.
\end{align*}
As $\psi'(x)=0$ has no solution in $(\gamma, 1),$  also $\psi'(x) \to -(3079/15552)<0$ when $x \to 1.$  Thus $\psi(x)$ is a decreasing function and hence $\psi(x)\leq \psi(\gamma) \approx 0.150413.$
In a nutshell, cases \textbf{I}-\textbf{II}, leads us to inequality \eqref{EQN33}. In Lemma \ref{l6}, on substituting $\zeta_{1}=\eta,\zeta_{2}=1$ and $\zeta_{3}=0$ in \eqref{EQN36} we obtain $p_{1}\in\mathcal{P}$ defined in \eqref{EQN38}. Thus  $\tilde{f}\in\mathcal{A}$ defined in \eqref{EQN46} belongs to $\mathcal{F}_{1}$ and serves as the extremal function.   
\end{proof} 

\begin{theorem}
Suppose $X(x)=-176 x^4-224 x ^3-264 x^2+328 x+469.$ Let $F_{f}$ be given by \eqref{EQN26}. 
If $f\in\mathcal{F}_{3},$ then sharp bound on Hankel determinant for $F_{f}$ is given by
\begin{equation}\label{EQN61}
    |H_{2,1}(F_{f})|\leq X(\eta)/2304,
\end{equation} where $\eta\approx 0.3737$ is the unique real root of the equation $X'(x)=0.$ This result is sharp for the function 
\[f(z)=\int_{0}^{z}p_{2}(t)/(1-t+t^{2}) dt,\]
where $p_{2}(z)$ is given by \eqref{EQN38}.
\end{theorem}
\begin{proof}
Suppose $f\in\mathcal{A}$ such that $(1-z+z^{2})f'(z)=p(z),$ where $p\in\mathcal{P}$ is defined in \eqref{EQN53}. Proceeding in the similar approach, we have
\[a_{2}=\frac{1}{2}(1+c_1),\quad a_{3}=\frac{1}{3} (c_1+c_2),\quad a_{4}=\frac{1}{4} (c_2+c_3-1)\]
The expression $\gamma_{1}\gamma_{3}-\gamma_{2}^{2}$  in terms of $\zeta_{i}'s$ $(i=1,2,3),$ becomes 
\begin{align}
   \mathcal{G}:= \gamma_{1}\gamma_{3}-\gamma_{2}^{2}&=\frac{1}{32} (1+2 \zeta _1) (2\zeta _1^3 (1-\zeta _2)^2 + 2  \zeta _1 \zeta _2(2-\zeta _2) +2 (\zeta _2 + \zeta _3)-1 \nonumber \\&\quad + 2 \zeta _1^2 (1 - \zeta _2-\zeta _3(1-| \zeta _2| {}^2))-2 \zeta _3 | \zeta _2| {}^2) +\frac{1}{768} (1+2 \zeta _1)^{4} \nonumber  \\ & \quad  - \frac{1}{9} (\zeta _1+\zeta _2+(1-\zeta _2) \zeta _1^2)^{2} \label{EQN58}.
\end{align}
\noindent \textbf{I.} In expression \eqref{EQN58} put $\zeta_{1}=1,$ then 
$|\mathcal{G}|=133/2304\approx 0.0577257.$ \\
\textbf{II.} If $\zeta_{1}\in [0,1),$ then from \eqref{EQN58} together with $|\zeta_{3}|\leq 1,$ gives
\begin{equation}\label{EQN59}
    |\mathcal{G}| \leq \frac{1}{16} (1+2 \zeta _1) (1-\zeta _1^2) \Psi(A,B,C),
\end{equation}
where  \[\Psi(A,B,C):=|A+B \zeta _2+C \zeta _2^2| + 1-| \zeta _2|^{2}\]
with
\[A:=\frac{80 \zeta _1^4+16 \zeta _1^3-40 \zeta _1^2-120 \zeta _1-69}{144 \left(1+2 \zeta _1\right) \left(1-\zeta _1^2\right)},\quad  B:=\frac{4 \zeta _1^2+4 \zeta _1+9}{9(1+2 \zeta _1)}\] and
\[C:=-\frac{2 \zeta _1^2+9 \zeta _1+16}{9(1+2 \zeta _1)}.\]

\noindent The expressions of $A,B$ and $C$ defined above leads to $AC>0.$ Now the inequality 
\begin{equation}\label{EQN43}
    | B| -2 (1-|C| )= \frac{8 \zeta _1^2-14 \zeta _1+23}{9(1+2 \zeta _1)}< 0,
\end{equation}
is equivalent to  $8 \zeta _1^2-14 \zeta _1+23 < 0.$ Since $\min_{\zeta_{1}\in[0,1)} 8 \zeta_1^2-14 \zeta _1+23 =135/8>0,$ then inequality \eqref{EQN43} is false for each $\zeta_{1}\in[0,1).$ Therefore $| B| -2 (1-|C| )\geq 0$ for each $0\leq \zeta_{1}<1.$
An application of Lemma \ref{l7} to inequality \eqref{EQN59}, gives
\begin{align}\label{EQN60}
    |\mathcal{G}|&\leq \frac{1}{16} (1+2 \zeta _1) (1-\zeta _1^2) (| A| +| B| +| C|) \leq S(\eta) \quad (\eta \approx  0.373776),
\end{align}  
where \begin{equation}
    S(x):=\frac{X(x)}{2304}=\frac{-176 x^4-224 x ^3-264 x^2+328 x+469}{2304}.
    \end{equation}
Note that $S'(x)$ vanishes at $x=\eta$ and $S''(\eta)<0,$ where $\eta \approx  0.373776$ is the only critical point of $S(x)$ in $[0,1).$ Hence \eqref{EQN60} determines the desired bound. On replacing $\zeta_{1}=\eta,\zeta_{2}= -1$ and $\zeta_{3}=0$ in \eqref{EQN36} we obtain $p_{2}\in\mathcal{P}$ defined in \eqref{EQN38}. As a result sharpness in \eqref{EQN61} is achieved by $\tilde{f}\in\mathcal{A}$ such that 
$(1-z+z^2)\tilde{f}'(z)=p_{2}(z).$ 
\end{proof}


\begin{theorem}
Suppose  $48 (17 + x) X(x)=(1 + x) (x^4+20 x^3-114 x^2+4 x+125).$  Let $F_{f}$ be given by \eqref{EQN26}. 
If $f\in\mathcal{F}_{4},$ then sharp bound on Hankel determinant for $F_{f}$ is given by 
\begin{equation}\label{EQN49}
|H_{2,1}(F_{f})|\leq X(\eta),
\end{equation} where $\eta \approx 0.381$ is the unique real root of the equation   $X'(x)=0.$ This result is sharp with extremal function 
\begin{equation}\label{EQN50}
    \tilde{f}(z)=\int_{0}^{z}p_{3}(t)/(1-t)^{2} dt,
\end{equation}
where $p_{3}(z)$ is defined in \eqref{EQN38}.
\end{theorem}
\begin{proof}
Let $f\in\mathcal{A}$ satisfy $(1-z+z^{2})f'(z)=p(z),$ where where $p\in\mathcal{P}$ is defined in \eqref{EQN53}. Following the same procedure as before, we have
\[a_{2}=\frac{1}{2} \left(c_1+2\right),\quad a_{3}=\frac{1}{3} \left(2 c_1+c_2+3\right),\quad a_{4}=\frac{1}{4} \left(3 c_1+2 c_2+c_3+4\right).\]
Due to Lemma \ref{l6}, the expression $\gamma_{1}\gamma_{3}-\gamma_{2}^{2}$ in terms of $\zeta_{i}'s$ $(i=1,2,3),$ becomes 
\begin{align}\label{EQN44}
   \mathcal{H}:= \gamma_{1}\gamma_{3}-\gamma_{2}^{2}&=
    \frac{1}{144} (18 (1+\zeta _1)^2 (\zeta _1^2 (1-\zeta _2){}^2 - \zeta _1 (\zeta _2^2 -1 + \zeta _3( 1 - | \zeta _2| {}^2))\nonumber \\& \quad -\zeta _3 | \zeta _2| {}^2+2 \zeta _2+\zeta _3+2)+3 (1 + \zeta _1){}^4-4 (3+4 \zeta _1+2 \zeta _2 \nonumber \\& \quad - 2  \zeta _1^2 (\zeta _2-1)){}^2).
\end{align}

\noindent \textbf{I}:  Substitute $\zeta_{1}=1$ in \eqref{EQN44}, then $|\mathcal{H}|=1/12.$\\
\textbf{II}: Since $|\zeta_{3}|\in{\overline{\mathbb{D}}},$ then for $\zeta_{1}\in [0,1),$ we have
\begin{equation}\label{EQN62}
    |\mathcal{H}| \leq \frac{1}{8}(1-\zeta _1) (1+\zeta _1)^{2} \Psi(A,B,C),
\end{equation}
where  \[\Psi(A,B,C):= | A + B \zeta _2 + C \zeta _2^2| + 1-| \zeta _2|^{2}\]
with 
\[A:=\frac{5 \zeta _1^4+2 \zeta _1^3-4 \zeta _1^2+6 \zeta _1+3}{18 \left(1-\zeta _1\right) \left(1 + \zeta _1\right){}^2},\quad  B:=-\frac{2 \left(1-\zeta _1\right) \left(3 + \zeta _1\right)}{9 \left(1 + \zeta _1\right)}\] and \[C:=-\frac{1}{9} \left( 8 + \zeta _1\right).\] 

\noindent The expressions of $A,B$ and $C$ defined above, demonstrates that $AC<0.$ Moreover, based on the conditions in Lemma \ref{l7}, we now look into the following subcases. \\
\indent \textbf{A.} Note that the inequality 
\[| B| -2 (1-|C| )= \frac{4(1 - \zeta _1)}{9 (1 + \zeta _1)}\leq 0\]
is  valid if $\zeta _1\geq 1.$ This is false as $\zeta_{1}<1.$ Infact
\[-4 A C \left(\frac{1}{C^2}-1\right) - B^2 = \frac{2 (\zeta _1^5+21 \zeta _1^4-10 \zeta _1^3-2 \zeta _1^2+93 \zeta _1-31)}{27 (1 + \zeta _1){}^2 (8+\zeta _1)}\leq 0,\] provided $\zeta_{1}\in[0,\zeta^{*}],$ where $\zeta^{*}\approx 0.336931.$ Consequently, due to Lemma \ref{l7}, this case does not hold for each $\zeta_{1}\in [0,1).$  \\
\indent \textbf{B.} Note that the following inequalities 
\[ \mathcal{J}_{1}:=4 (1 + |C|)^2=\frac{4}{81} ( 17 + \zeta _1){}^2>0 \]and 
\[\mathcal{J}_{2}:=-4 A C \left(\frac{1}{C^2}-1\right) = \frac{2 (17 + \zeta _1) (5 \zeta _1^4+2 \zeta _1^3-4 \zeta _1^2+6 \zeta _1+3)}{81 (1 + \zeta _1){}^2 (8 + \zeta _1)}>0,\]
hold simultaneously if $\zeta_{1}> \zeta^{*}.$ Thus 
\begin{align*}
B^2=\frac{4 \left(1-\zeta _1\right){}^2 \left(3+\zeta _1\right){}^2}{81 \left(1+\zeta _1\right){}^2} & <\min \left\{\mathcal{J}_{1},\mathcal{J}_{2}\right\}=\mathcal{J}_{2}.
\end{align*}
Therefore in view of Lemma \ref{l7}, inequality \eqref{EQN62}, gives
\begin{align*}
    |\mathcal{H}| & \leq \frac{1}{8} \left(1-\zeta _1\right) \left(1+\zeta _1\right)^{2} \left(1+\left| A\right| +\frac{B^2}{4 (1 + \left| C\right|)}\right) \leq X(\eta) \quad (\eta \approx 0.381423),
\end{align*}
where \[X(x)=\frac{(1 + x) (x^4+20 x^3-114 x^2+4 x+125)}{48 (17 + x)}.\]
Since $x=\eta > \zeta^{*}$ is a unique real root of $X'(x)=0$ in $(\zeta^{*},1)$ and $X''(\eta)<0,$ where $\eta \approx 0.381423$. Then $X$  attains its maximum at $x=\eta.$ \\
\indent \textbf{C.} Further for $0\leq \zeta_{1} \leq \zeta^{*},$ simple computations reveal that 
\begin{align*}
|C| (|B| + 4 |A|) - |A B|&=\frac{17 \zeta _1^6+128 \zeta _1^5+137 \zeta _1^4-80 \zeta _1^3-17 \zeta _1^2+160 \zeta _1+87}{81 \left(1-\zeta _1\right) \left(1 + \zeta _1\right){}^3} \\& \geq m >0 \quad (m \approx 1.05323). 
\end{align*}
Similarly $\left| A B\right| -\left| C\right|  (\left| B\right| -4 \left| A\right| )>0.$ Therefore, Lemma \ref{l7}, together with \eqref{EQN62}, establishes the following
\begin{align*}
    |\mathcal{H}| & \leq \frac{1}{48} (3 \zeta _1^4-16 \zeta _1^3-18 \zeta _1^2+24 \zeta _1+19) \\& \times  \sqrt{\frac{\zeta _1^5+12 \zeta _1^4+8 \zeta _1^3-2 \zeta _1^2+3 \zeta _1+14}{15 \zeta _1^5+126 \zeta _1^4+36 \zeta _1^3-78 \zeta _1^2+153 \zeta _1+72}} \leq n \quad (n \approx 0.183792).
\end{align*}
By summarizing all the above cases, the required bound  can be established. Here equality in \eqref{EQN49} holds if we replace $\zeta_{1}= \eta,\zeta_
{2}\approx-0.0871127$ and $\zeta_{3}= 1$ in \eqref{EQN44}. On substituting these values of $\zeta_{i}$ $'s$ $(i=1,2,3)$ in \eqref{EQN36} we obtain $p_{3}\in\mathcal{P}$ defined in \eqref{EQN38}. Infact $\tilde{f}\in\mathcal{A}$ given in \eqref{EQN50} satisfies $(1-z)^{2}\tilde{f}'(z)=p_{3}(z).$ Hence $\tilde{f}\in\mathcal{F}_{4}$  serves as the extremal function. 
\end{proof}

\end{document}